\documentclass[sn-basic]{sn-jnl}
\pdfoutput=1

\usepackage{fullpage}
\usepackage{amsmath,amssymb,amsthm}
\usepackage{graphicx}
\usepackage{algorithm,algorithmic}
\usepackage{hyperref}

\newtheorem{proposition}{Proposition}
\theoremstyle{definition}
\newtheorem{example}{Example}

\newcommand{\revisionone}[1]{#1}

\DeclareMathOperator{\tr}{tr}
\DeclareMathOperator{\prox}{prox}

\title{Defining Lyapunov functions as the solution of a performance estimation saddle point problem}

\author{\fnm{Olivier} \sur{Fercoq}}\email{olivier.fercoq@telecom-paris.fr}

\affil{\orgdiv{LTCI}, \orgname{Télécom Paris, Institut Polytechnique de Paris}, 
	
\orgaddress{\street{19 place Marguerite Perey}, \city{Palaiseau}, \postcode{91120}, \country{France}}}

\begin{document}

\abstract{
In this paper, we reinterpret quadratic Lyapunov functions as solutions to a performance estimation saddle point problem. This allows us to automatically detect the existence of such a Lyapunov function and thus numerically check that a given algorithm converges. The novelty of this work is that we show how to 
define the saddle point problem using the PEPit software  and
then solve it with DSP-CVXPY.

This combination gives us a very strong modeling power because defining new points and their relations across iterates is very easy in PEPit. We can without effort define auxiliary points used for the sole purpose of designing more complex Lyapunov functions, define complex functional classes like the class of convex-concave saddle point problems whose smoothed duality gap has the quadratic error bound property or study complex algorithms like primal-dual coordinate descent method.
}


\maketitle

\section{Introduction}

Performance estimation problems (PEP), introduced in \cite{drori2014performance}, are optimization problems whose goal is
to study the worst case performance of optimization algorithms.
When combined with interpolation results on the functional class of
interest, they can be used to show tighter bounds and design new, faster algorithms~\cite{taylor2017convex}. 
However, the optimization algorithms that we get with PEP have a number
of variables and constraints that grow linearly with the number of iterations we wish to consider. Since they are eventually solved 
using semi-definite programming solvers, this can be quickly a challenging task.

In order to study the convergence of complex algorithms, a widely used 
technique is to exhibit a Lyapunov function, which is a nonnegative function whose value is strictly decreasing at each iteration. It has been shown in~\cite{upadhyaya2024automated} that performance estimation problems can be used to automatically detect the presence of a quadratic Lyapunov function and thus numerically check the convergence of an algorithm. This was used for instance to show that the Chambolle-Pock algorithm converges for a wider range of step sizes than previously known~\cite{banert2023chambolle}.

In this paper, we reinterpret quadratic Lyapunov functions as the solutions to a performance estimation saddle point problem. Like in \cite{upadhyaya2024automated}, this allows us to automatically detect the existence of such a Lyapunov function and thus numerically check that a given algorithm converges. The novelty of this work is that we show how to 
define the saddle point problem using the PEPit software \cite{goujaud2024pepit} and
then solve it with DSP-CVXPY \cite{schiele2023disciplined}.
This combination gives us a very strong modeling power because defining new points and their relations across iterates is very easy in PEPit.

We conducted three numerical experiments to show the interest of our contribution. Firstly, we show that we can define complex functional classes and analyze the worst case behavior of algorithms on this class. We considered the class of convex-concave saddle point problems whose smoothed duality gap has the quadratic error bound property \cite{fercoq2023quadratic} and showed that the primal-dual hybrid gradient method~\cite{chambolle2011first} converges linearly under this assumption. More precisely, our study suggests that the fastest rate is not attained with the largest step sizes.

Secondly, we show that we can numerically study a randomized algorithm, namely randomized coordinate gradient descent~\cite{nesterov2012efficiency}. Indeed, if we try to estimate all possible outcomes of the algorithm, we quickly face an explosion in the number of possibilities. Thanks to this study, we could find a slightly improved worst case bound on the expected
objective value.

Finally, we show that we can study complex algorithms like the primal-dual coordinate descent method \cite{alacaoglu2020random}.
In that case, we reached the limits of the approach, with solvers struggling to return accurate solutions and the solutions returned difficult to interpret.

\revisionone{
\section{Notation}
\label{sec:notation}

In this paper, we are going to study the convergence of optimization algorithms using Lyapunov functions.
We shall consider a class of functions $\mathcal F$ satisfying some properties like convexity or smoothness. 
A function $f \in \mathcal F$ will be real valued $f: \mathcal X \to \mathbb R$ and we will consider
a list of points $x = (\xi_i)_{1 \leq i \leq p} \in \mathcal X^p$ on which an algorithm $A$ may act.

Those $p$ points, linked to the concept of leaf points in PEPit \cite{goujaud2024pepit}, can be considered as intermediate points in the construction of the next iteration. We will assume that they are not linearly dependent but they may depend one on the other 
with relation that involve the function $f$.
Then we will denote $A(f, x) \in \mathcal X^p$ the next iterate, which is the transformation of each leaf point $\xi_i$ into a new vector of $\mathcal X$. Of course the way each iterations proceeds depends on the function $f$.

For instance, to encode the fixed-step gradient descent algorithm on the class $\mathcal F_{0,L}$ of convex functions with $L$-Lipschitz gradient, we shall take $\xi_1 \in \mathcal X$, $\xi_2 = \nabla f(\xi_1)$, $\xi_3 \in \arg\min f$ and 
\begin{equation*}
A(f, [\xi_1, \xi_2, \xi_3]) = [\xi_1 - \gamma \xi_2, \nabla f(\xi_1 - \gamma \xi_2),\xi_3] \;.
\end{equation*}

By an abuse of notation, for a matrix $Q \in \mathbb R^{p \times p}$ and a list of vectors $x \in \mathcal X^p$, we shall denote $x^\top Q x = \sum_{i=1}^p\sum_{j=1}^p Q_{i,j} \langle x_i, x_j\rangle$.
}

\section{From Lyapunov function to saddle point problem}

A widely used technique to prove convergence of an algorithm is to find a Lyapunov function $V \geq 0$ such that
\begin{align}
\label{eq:def_lyap}
V(f, A(f, x)) + R(f, x) \leq V(f, x) \qquad \forall f \in \mathcal F, \forall x \in \mathcal X
\end{align}
where $A(f,x)$ is the outcome of one iteration of algorithm $A$ applied on function $f$ at the point $x$. In order to prove linear convergence, we are looking for a real number $0 < \rho < 1$ and a function $V$ such that 
\begin{align}
R(f, A(f,x)) \leq V(f, A(f, x)) \leq \rho V(f, x) \qquad \forall x \in \mathcal X, \forall f \in \mathcal F
\end{align}

\begin{example}
Consider the gradient descent method on the class of differentiable convex functions whose gradient is $L$-Lipschitz continuous. We can choose:
\begin{itemize}
	\item $\mathcal X = \mathbb R^n$,
	\item $f \in \mathcal F$ if and only if for all $(x^1,x^2) \in \mathcal X, f(x^2) \geq f(x^1) + \langle \nabla f(x^1), x^2 - x^2 \rangle + \frac{1}{2L} \|\nabla f(x^1)-\nabla f(x^2)\|^2$,
	\item \revisionone{For $\xi_1 \in \mathcal X$, $\xi_2 = \nabla f(\xi_1)$, $\xi_3 \in \arg\min f$, define $A(f,x) = A(f, [\xi_1, \xi_2, \xi_3]) = [\xi_1 - \gamma \xi_2, \nabla f(\xi_1 - \gamma \xi_2),\xi_3]$} for some step size parameter $\gamma \in \mathbb R$,
	\item $R(f, x) = \revisionone{f(\xi_1) - f(\xi_3)}$.
\end{itemize}
If we are able to find a Lyapunov function $V$ then sublinear convergence of the function value gap will follow by standard arguments~\cite{nesterov2013introductory}.
\revisionone{
Indeed, we define $x_{k+1} = A(f, x_k)$ and we sum \eqref{eq:def_lyap}:
\begin{align*}
&V(f, x_{k+1}) + R(f, x_k) \leq V(f, x_k) \\
&V(f, x_K) + \sum_{k=0}^{K-1} R(f,x_k) \leq V(f, x_0)
\end{align*}
Since $V(f, X_K) \geq 0$, this implies that $\sum_{k=0}^{+\infty} R(f,x_k) < +\infty$ and so $\lim_{k \to +\infty} R(f,x_k) = 0$. This can also be used to get finite-time convergence rates.
}
\end{example}

Let us now present the main derivation of this paper. We shall consider the set $\mathcal Q$ of quadratic Lyapunov functions parameterized as 
\begin{equation*}
V(f, x) = x^\top Q x + q^\top N(f, x)
\end{equation*}
for some positive semi-definite matrix $Q \in \mathbb S^p_+$, some nonnegative vector $q$, some minimizer $x^*$ of $f$ and nonnegative quantities $N(f, x)$ that depend on the problem only. \revisionone{Here we abuse notation as explained in Section~\ref{sec:notation}.}

\begin{align}
&\exists V \in \mathcal Q \text{ s.t. } V(f, A(f, x)) + R(f, x) \leq V(f, x) \qquad \forall f \in \mathcal F, \forall x \in \mathcal X^{\revisionone{p}}
\label{lyap-def}\\
&\Leftrightarrow \quad 
\exists V \in \mathcal Q \text{ s.t. } \max_{f \in \mathcal F, x \in \mathcal X^p} V(f, A(f, x)) + R(f,x) - V(f,x) \leq 0 
\label{lyap-max}\\
&\Leftrightarrow \quad
\min_{\{V \in \mathcal Q\}} \max_{\{f \in \mathcal F, x \in \mathcal X^p\}} V(f, A(f, x)) + R(f, x) - V(f, x) \leq 0
\label{lyap-minmax}
\end{align}

At this point, this does not really help. But the inner maximization problem fits into the framework of performance estimation problems \cite{drori2014performance,taylor2017convex}.
\revisionone{Using this fact, in the next section, we are going to show how the search for a quadratic Lyapunov function can be written as a bilinear convex-concave saddle point problem involving semi-definite variables.}

\section{Constructing the transition matrix using PEPit}

\revisionone{
Consider the problem 
\begin{align*}
\max_{\{f \in \mathcal F, x \in \mathcal X^p\}}  R(f, x)
\end{align*}
As shown in~\cite{taylor2017convex}, it can be written equivalently as
\begin{align*}
&\max_{G \succeq 0, F} c^\top F + \langle C, G\rangle \\
&\text{s.t. } a_i + b_i^\top F + \langle D_i, G \rangle \leq 0 , \forall i \in I
\end{align*}
where $G$ is a semi-definite optimization variable, $F$ is a vectorial optimization variable and $c, C, a_i, b_i, D_i$ are given 
scalars, vectors and matrices that encode the function $R$ and the constraints defining $\mathcal F$ and the steps of the algorithm $A$.
Suppose that one iteration of the algorithm can be described using $\bar p$ leaf points $(\xi_1, \ldots, \xi_{\bar p})$. Then the matrix $G$ is called the Gram matrix and is such that $G_{i,j} = \langle \xi_i, \xi_j\rangle$ for all $j \in \{1, \ldots, \bar p\}$.
The change of variable from $x = (\xi_1, \ldots, \xi_{\bar p})$ to $G$ is surjective as soon as $\dim(\mathcal X) \geq \bar p$, as can be seen thanks to a Cholesky factorization.
With this matrix, we can rewrite $x^\top Q x = \sum_{i,j} Q_{i,j} \langle \xi_i, \xi_j\rangle = \tr(Q G)$.

Let us now consider a vector $y \in \mathcal X$. 
When defined in PEPit, any such vector is either a new leaf point
or a linear combination of previously defined leaf point. In the latter case, PEPit stores the leaf point representation of $y$.
In this paper, we shall denote $\pi(y)$ this leaf point representation, which means that $y = \sum_{k=1}^{\bar p} \pi(y)_k \xi_k$. In particular, $\pi(\xi_i) = e_i$, the $i$th canonical basis vector.

We now define the matrix $\Sigma$, which we shall call the transition matrix, by
\begin{equation*}
\Sigma \pi(\xi_i) = \pi((A(f, (\xi_j)_{1 \leq j \leq p}))_i) \;.
\end{equation*}
for all $i \in \{1, \ldots, p\}$. In~\cite{upadhyaya2024automated}, a similar transition matrix is defined. More precisely it is the transpose of what we present here.
Let us have a look at how the Gram matrix is transformed when one iteration takes place:
\begin{align*}
G^+_{i,j} & = \langle (A(f,x))_i, (A(f, x))_j\rangle = \sum_{k=1}^{\bar p} \sum_{l=1}^{\bar p} \pi(A(f,x))_i)_k\pi((A(f, x))_j)_l\langle \xi_k, \xi_l \rangle = \sum_{k=1}^{\bar p} \sum_{l=1}^{\bar p} \Sigma_{k,i} \Sigma_{l,j}G_{k,l} \\
G^+& = \Sigma^\top G \Sigma
\end{align*}
Finally, we get 
\begin{align*}
A(f,x)^\top Q A(f,x) = \sum_{i,j} Q_{i,j} \langle (A(f,x))_i, (A(f, x))_j\rangle = \sum_{i,j} Q_{i,j} G_{i,j}^+ = \tr(Q \Sigma^\top G \Sigma) \;.
\end{align*}
Using similar ideas, we can define a transition matrix $\sigma$ for leaf expressions, such that if $N(f, x) = N F$
then $N(f, A(f,x)) = N \sigma F$.

Gathering these we can write \eqref{lyap-minmax} as a saddle point
problem with semi-definite variables:
}
\begin{align*}
&\exists V \in \mathcal Q \text{ s.t. } V(A(f, x), x^*) + R(x, x^*) \leq V(x, x^*) \\
&\Leftrightarrow \quad
\revisionone{\min_{\{Q \succeq 0, q \geq 0\}} \max \left\{\begin{matrix} \tr(Q \Sigma^\top G \Sigma) + q^\top N \sigma F + c + \tr(C G) - \tr(Q G) - q^\top N F \hfill \\
\text{subject to: } a_i + b_i^\top F + \tr(D_iG) \leq 0, \quad \forall i \in I \end{matrix}\right\} }\leq 0
\end{align*}

The constraints of the inner problem are positive semi-definiteness, linear inequality and linear equality constraints that are automatically built by the PEPit software~\cite{goujaud2024pepit}.
We are left with a convex-concave saddle point problem which can be solved by the off-the-shelf solver DS-CVXPY~\cite{schiele2023disciplined}.

We now show how PEPit can help us implement the transition matrix $\Sigma$. When we define a new point in PEPit, two cases can occur. It is either a linear combination of previously defined points or it is what is called a leaf point.
Defining a new leaf point means that a new variable is required in the performance estimation problem. It is then associated with constraints that will ensure that it truly represents what it is meant to represent.

For instance, when defining the gradient descent method on a function $f$, 
we first define $x^* = 0$, where $\nabla f(x^*) = 0$ and an initial point $x$, which will be a leaf point with constraint $\|x - x^*\|^2 \leq R$. 
Then, in order to represent $x^+= x - \gamma \nabla f(x)$, we need a second 
leaf point, which is $\nabla f(x)$. PEPit automatically defines all the 
constraints that are necessary to ensure that $\nabla f(x)$ is interpolable with a function of the class chosen for $f$.
The new iterate $x^+$ is not a leaf point because it can be recovered as a linear combination of other points.
In order to illustrate the construction of the matrix $\Sigma$, suppose $R(x) = \|\nabla f(x^+)\|^2$, so that we are also interested in $\nabla f(x^+)$.

The Gram matrix $G$ is such that 
\[
G = \begin{bmatrix}
x^\top x & x^\top \nabla f(x) & x^\top \nabla f(x^+) \\
\nabla f(x)^\top x & \nabla f(x)^\top \nabla f(x) & \nabla f(x)^\top \nabla f(x^+) \\
\nabla f(x^+)^\top x & \nabla f(x^+)^\top \nabla f(x) & \nabla f(x^+)^\top \nabla f(x^+)
\end{bmatrix} \; .
\]
We take $p = 2$ (the number of leaf points we consider in our Lyapunov function) and we have $\bar p = 3$ (the number of leaf points constructed by PEPit to describe the algorithm).
Then since $\pi(x) = [1, 0, 0]$ and $\pi(\nabla f(x)) = [0,1,0]$,
$\pi(x^+) = [1, -\gamma, 0]$. We also have $\pi(\nabla f(x^+)) = [0, 0, 1]$. Hence,
\[
\Sigma = \begin{bmatrix}
1 & 0 & 0 \\
-\gamma & 0 & 0 \\
0 & 1 & 0
\end{bmatrix}
\]
We can compute 
\[
\Sigma^\top G \Sigma = \begin{bmatrix}
(x^+)^\top x^+ & (x^+)^\top \nabla f(x^+) & 0 \\
\nabla f(x^+)^\top x^+ & \nabla f(x^+)^\top \nabla f(x^+) & 0 \\
0 & 0 & 0
\end{bmatrix} \;.
\]
In order to construct $\Sigma$, we consider the leaf points where the Lyapunov function $V$ is going to be computed and their images after applying the algorithm. Let $u$ be such a leaf point ($u \in \{x, \nabla f(x)\}$ for the gradient descent case). If we consider the decomposition dictionary
of the product $u^\top u^+$, then we directly obtain the row of $\Sigma$ corresponding to $u$. This decomposition dictionary $W$ is a matrix of the same size as $G$ such that $\tr(W^\top G) = u^\top u^+$. It is implemented in PEPit, so directly usable. Note that we could also implement the matrix $\Sigma$ from the decomposition dictionary of $u^+$ in leaf points but the proposed solution allows one to directly manipulate matrices.

\section{Numerical experiments}

We considered three numerical experiments that show how the technique of automatic Lyapunov functions may be used to better understand optimization algorithms.
The code for the numerical experiments is available on \href{https://perso.telecom-paris.fr/ofercoq/Software.html}{https://perso.telecom-paris.fr/ofercoq/Software.html}

\subsection{Linear convergence of PDHG with longer steps under quadratic error bound of the smoothed gap}

Consider the Primal-Dual Hybrid Gradient (PDHG) algorithm designed to solve the saddle point problem
\[
\min_{x\in \mathcal X} \max_{y \in \mathcal Y} f(x) + \langle Mx, y \rangle - g(y) \;.
\]
where $f$ and $g$ are be two convex functions whose proximal operator is easy to compute
and $M$ is a linear operator.
PDHG is the algorithm
\begin{align*}
&\bar y_{k+1} = \prox_{\sigma g}(y_k + \sigma M x_k) \\
&x_{k+1} = \bar x_{k+1} = \prox_{\tau f}(x_k - \tau M^\top \bar y_{k+1}) \\
&y_{k+1} = \bar y_{k+1} + \sigma M(\bar x_{k+1} - x_k)
\end{align*}

It has been shown in \cite{fercoq2023quadratic} that PDHG converges linearly under the assumption that the smoothed duality gap has the quadratic error bound property.
This proof of linear convergence requires $\gamma < 1$ where $\gamma = \sigma \tau \|M\|^2$, $\sigma$ is the dual step size, $\tau$ is the primal step size. 
Nevertheless, the algorithm does converge for $\gamma = \sigma \tau \|M\|^2 < 4/3$ \cite{li2022improved}. 
In the experiment summarized in Figure~\ref{fig:pdhg_qebsm_gamma}, we numerically extend the analysis of \cite{fercoq2023quadratic} to the full range $0 < \gamma < 4/3$ for the case $\tau = \gamma = \frac{\sqrt{\gamma}}{\|M\|}$. 
Of course, the result may be subject to numerical instabilities and we do not have a closed formula for the rate, but it has been obtained without having to manipulate a complex combination of inequalities.

\begin{figure}
	\centering
	
	\includegraphics[width=0.7\linewidth]{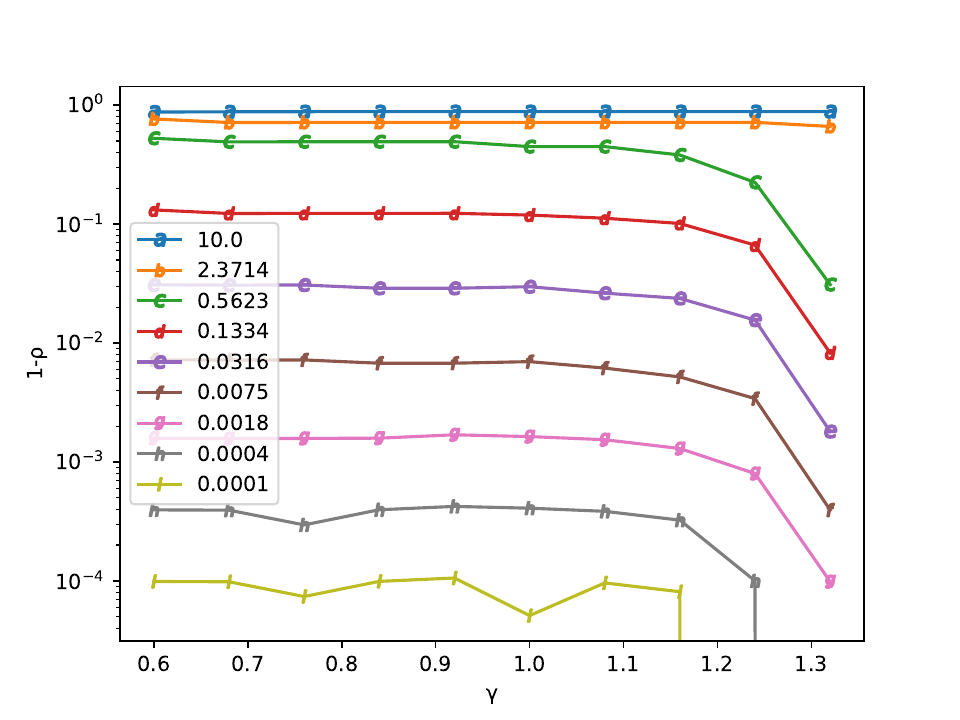}
	\caption{Rate of convergence of PDHG as a function of $\gamma$ ($1-\rho$ close to $10^0$ means very fast convergence). Each curve corresponds to a given value for the quadratic error bound constant for the smoothed gap with parameter $\beta=1$.
	Note that as shown in \cite{li2022improved}, $\gamma$ can be chosen up to 4/3. However, in terms of rate of convergence, this figure suggests not to try approaching too much the convergence frontier.}
\label{fig:pdhg_qebsm_gamma}
\end{figure}

\subsection{Better Lyapunov function for coordinate descent}

In this experiment, we show how to use our results for a randomized algorithm. We chose coordinate gradient descent. We consider a convex differentiable function $f$ whose $i$th partial derivative is $L_{i}$-Lipschitz continuous. Suppose we have $d$ blocks of coordinates. The algorithm is 
\begin{align*}
&\text{Generate }I_{k+1} \sim U(\{1, \ldots, d\}) \\
& x_{k+1} = x_k - \frac{1}{L_{I_{k+1}}} \nabla_{I_{k+1}} f(x_k) U_{I_{k+1}}
\end{align*}
where $U_i$ is the submatrix of the identity corresponding to block $i$.

In~\cite{lu2015complexity}, it is shown that for all $k$,
\begin{align*}
\mathbb E[d(f(x_{k+1}) - f(x_*)) + \frac{d}{2} \|x_{k+1} - x_*\|^2_L | x_k] \leq (d-1)(f(x_k) - f(x_*)) + \frac{d}{2} \|x_{k} - x_*\|^2_L
\end{align*}
Our goal is to recover this result, and perhaps improve it.

Since it is a randomized algorithm, we need to consider all the possible outcomes of the random process. Fortunately, the Lyapunov analysis amounts to studying only one iteration, from $x_0$ to $x_1$, so that there are only $d$ possible outcomes.
We construct $d$ possible values for $x_{1}(I_{1} = i)$ and $d$ transition matrices $\Sigma(I_{1} = i)$. The final Lyapunov function is then the average of all the possible outcomes.

Using our automatic Lyapunov technique, we considered this algorithm for the case $L_{i} = 1$ for all $i$ and various values for $d$. We then obtain nonnegative numbers $q_1$ and $q_2$ such that for all $f$ satisfying the assumptions, 
\begin{align*}
\mathbb E[q_1 (f(x_1) - f(x_*)) + q_2 \|x_1 - x_*\|^2_L] \leq (q_1-1) (f(x_0) - f(x_*)) + q_2 \|x_0 - x_*\|^2_L
\end{align*}
We also tried considering additional terms involving $\nabla f(x_0)$ but this did not improve the results.
In Table~\ref{tab:rcd} we can see what the saddle point problem resolution yields.
\begin{table}
	\centering
\begin{tabular}{|l|l|l|}
\hline
$d$ & $q_1$ & $q_2$ \\
\hline
2 & 1.000 & 0.999 \\
3 & 1.997 & 1.499 \\
4 & 3.000 & 2.000 \\
5 & 3.999 & 2.499 \\
8 & 7.005 & 4.002 \\
10 & 8.997 & 4.999 \\
\hline
\end{tabular}
\caption{Automatically determined Lyapunov function parameters for coordinate gradient descent.}
\label{tab:rcd}
\end{table}

From this table, we can conjecture a slightly improved Lyapunov inequality, namely
\begin{align}
\mathbb E[(d-1)(f(x_1) - f(x_*)) + \frac d2 \|x_1 - x_*\|^2_L] \leq (d-2) (f(x_0) - f(x_*)) + \frac d2 \|x_0 - x_*\|^2_L \;.
\label{conj:cd}
\end{align}

\begin{proposition}
The conjecture \eqref{conj:cd} is true for all $d \geq 1$.
\end{proposition}
\begin{proof}
We are first going to show that, in the weighted norm $\|x\|_L = (\sum_{i=1}^n L_{i} (x^{(i)})^2)^{1/2}$, $\nabla f$ is $d$-Lipschitz continuous.
Since $h = \frac{1}{d} \sum_{i=1}^d d h^{(i)} e_i$, we have
\begin{align*}
f(x+h) &\leq \frac{1}{d} \sum_{i=1}^d f(x+d h^{(i)} e_i)  \leq \frac{1}{d} \sum_{i=1}^d \Big(f(x) + d \nabla_{I} f(x) h^{(i)} + \frac{L_{I}}{2}d^2 (h^{(i)})^2\Big) = f(x) + \langle \nabla f(x), h \rangle + \frac{d}{2} \|h\|_L^2
\end{align*}
where, for all $i$, we used the fact that $\nabla_{I} f$ is coordinate wise Lipschitz continuous. Let us fix $x_0$ and define $\phi: x \mapsto f(x) - \langle\nabla f(x_0), x-x_0\rangle$. The function $\phi$ is convex and $\nabla \phi(x_0) = 0$. Using the elementwise division in $\frac{\nabla \phi(x)}{L}$, we have
\begin{align*}
f(x_0) &= \phi(x_0) \leq \phi\big(x - \frac{1}{d} \frac{\nabla \phi(x)}{L} \big)
= f\big(x - \frac{1}{d} \frac{\nabla f(x)-\nabla f(x_0)}{L}\big) - \big\langle \nabla  f(x_0), x - \frac{1}{d} \frac{\nabla f(x)-\nabla f(x_0)}{L} - x_0 \big\rangle \\
&\leq f(x) - \frac 1d \langle \nabla f(x),  \frac{\nabla f(x)-\nabla f(x_0)}{L}\rangle + \frac{d}{2d^2} \Big\|\frac{\nabla f(x)-\nabla f(x_0)}{L}\Big\|^2_L \\
& \quad\qquad+  \frac 1d \langle \nabla f(x_0),  \frac{\nabla f(x)-\nabla f(x_0)}{L}\rangle - \langle \nabla f(x_0), x - x_0\rangle 
\end{align*}
Hence, rearranging, we get
\begin{align}
f(x) \geq f(x_0) + \langle \nabla f(x_0) , x-x_0\rangle + \frac {1}{2d} \|\nabla f(x)-\nabla f(x_0)\|^2_{L^{-1}} \;.
\label{eqn:cocoercivity_d}
\end{align}
By summing this inequality and the one where we exchange the roles of $x$ and $x_0$, we obtain
\begin{align*}
\|\nabla f(x)-\nabla f(x_0)\|^2_{L, *} &= \|\nabla f(x)-\nabla f(x_0)\|^2_{L^{-1}} \leq d \langle \nabla f(x) - \nabla f(x_0) , x-x_0\rangle \\
&\leq d \|\nabla f(x)-\nabla f(x_0)\|_{L, *} \|x - x_0\|_L
\end{align*}
which leads to $\|\nabla f(x)-\nabla f(x_0)\|_{L, *} \leq d \|x - x_0\|_L$.

We can now proceed with the rest of the proof, using the notation $I=I_1$.
\begin{align*}
&(d-1) (f(x_1) - f(x_*)) \leq (d-1) (f(x_0) - f(x_*) + \langle \nabla_{I} f(x_0), x^{(I)}_1 - x^{(I)}_0 \rangle + \frac{L_{I}}{2} \|x^{(I)}_1 - x^{(I)}_0\|^2) \\
& = (d-1) (f(x_0) - f(x_*)) + d\Big(\langle \nabla_{I} f(x_0), x_*^{(I)} - x_0^{(I)} \rangle + \frac{L_{I}}{2} \|x_*^{(I)} - x_0^{(I)}\|^2 - \frac {L_{I}}2 \|x_*^{(I)} - x_1^{(I)}\|^2\Big)  + \frac{1}{2 L_{I}} \|\nabla_{I} f(x)\|^2
\end{align*}
where we used the fact that for all $x^*$, 
\begin{align*}
-\frac{1}{2 L_{I}} \|\nabla_{I} f(x)\|^2 & = \langle \nabla_{I} f(x_0), x^{(I)}_1 - x^{(I)}_0 \rangle + \frac{L_{I}}{2} \|x^{(I)}_1 - x^{(I)}_0\|^2\\
& = \langle \nabla_{I} f(x_0), x_*^{(I)} - x_0^{(I)} \rangle + \frac{L_{I}}{2} \|x_*^{(I)} - x_0^{(I)}\|^2 - \frac {L_{I}}2 \|x_*^{(I)} - x_1^{(I)}\|^2 \;.
\end{align*}
Since $x_*$ and $x_0$ do not depend on $I$, we can apply the expectation and get
\begin{align*}
(d-1) &\mathbb E[f(x_1) - f(x_*)] \\
&\leq (d-1) (f(x_0) - f(x_*)) + \langle \nabla f(x_0), x_* - x_0 \rangle + d\mathbb E\big[ \frac{1}{2} \|x_* - x_0\|^2_L - \frac {1}{2}\|x_* - x_1\|^2_L\big] + \frac{1}{2 d} \|\nabla f(x)\|_{L^{-1}}^2 \\
& \leq (d-2) (f(x_0) - f(x_*)) + d\mathbb E\big[ \frac{1}{2} \|x_* - x_0\|^2_L - \frac {1}{2}\|x_* - x_1\|^2_L\big]
\end{align*}
where we used \eqref{eqn:cocoercivity_d} with $x = x_*$.
\end{proof}

\subsection{Primal-dual coordinate descent}

As a final experiment, we show that we can deal with a complex algorithm, namely PURE-CD~\cite{alacaoglu2020random}, which is a primal-dual algorithm with randomized coordinate updates. It has been designed to solve problems of the form
\begin{equation*}
\min_{x\in\mathcal{X}_1 \times \ldots \times \mathcal X_n} \max_{y \in \mathcal Y} f(x) + \sum_{i=1}^{n}g_i(x_i) +  \langle Mx, y \rangle - h^*(y),
\end{equation*}
$\mathcal{X}$ and $\mathcal{Y}$ are Euclidean spaces such that $\mathcal{X} = \prod_{i=1}^n \mathcal{X}_i$, and $\mathcal{Y}=\prod_{j=1}^m \mathcal{Y}_j$.
The functions $f\colon \mathcal{X} \to \mathbb{R}$, $g_i \colon \mathcal X_i \to \mathbb R \cup \{+\infty\}$ and $h\colon \mathcal{Y} \to \mathbb{R}\cup\{+\infty\}$ are proper, lower semicontinuous, and convex, $M\colon \mathcal{X} \to \mathcal{Y}$ is a linear operator.
Moreover, $f$ is assumed to have coordinatewise Lipschitz continuous partial derivatives with constants $\beta_i$ and $g, h$ admit easily computable proximal operators.

The iterations of PURE-CD are given by Algorithm~\ref{alg:stripd}. Given positive probabilities $(p_i)_{1 \leq i \leq n}$, the algorithm requires the notation $J(i) = \{ j\in\{1, \dots, m\}\colon M_{j, i} \neq 0 \}$, $I(j) = \{ i\in\{1, \dots, n\}\colon M_{j, i} \neq 0 \}$, $\pi_j = \sum_{i \in I(j)} p_i$, $\underline p = \min_i p_i$, $\theta_j = \frac{\pi_j}{\underline p}$ and step size parameters $\tau_i$ and $\sigma_j$ such that
\begin{equation}
\tau_i < \frac{2p_i - \underline{p}}{\beta_i p_i + \underline p^{-1}p_i\sum_{j=1}^m \pi_j \sigma_j M_{j, i}^2}. \label{purecd_stepsizes}
\end{equation}

\begin{algorithm}[h]
	\caption{Primal-dual method with random extrapolation and coordinate descent (PURE-CD)}
	\label{alg:stripd}
	\begin{algorithmic}[1]
		\FOR{$k = 0,1\ldots $} 
		\STATE $\bar{y}_{k+1} = \prox_{\sigma, h^\ast}\left(y_k + \sigma Mx_k\right)$
		\STATE $\bar{x}_{k+1} = \prox_{\tau, g}\left(x_k - \tau \left(\nabla f(x_k) + M^\top \bar{y}_{k+1}\right)\right)$
		\STATE Draw $i_{k+1} \in \{ 1, \dots, n \}$ with $\mathbb{P}(i_{k+1} = i) = p_i$
		\STATE $x_{k+1}^{i_{k+1}} = \bar{x}_{k+1}^{i_{k+1}}$ 
		\STATE $x_{k+1}^{j} = x_k^j, \forall j \neq i_{k+1}$
		\STATE $y_{k+1}^{j} = \bar{y}_{k+1}^j + \sigma_j \theta_j(M (x_{k+1} - x_k))_j, \forall j \in J(i_{k+1})$, $y_{k+1}^j = y_k^j, \forall j\not\in J(i_{k+1})$
		\ENDFOR
	\end{algorithmic}	
\end{algorithm}

We numerically checked the result of \cite{alacaoglu2020random} by looking for a Lyapunov function associated to the algorithm. We considered the case where $f=0$, $n = 2$, $m=1$ and $p = (1/2, 1/2)$, so that $\pi = \pi_1 = 1$ and $\theta = \theta_1 = 2$. We also fixed the condition on $M$ that $\|M_{1, i}\|^2 \leq 1$ for all $i$. This is slightly more general than PDHG and already involves randomness in the algorithm.
In this setting, \eqref{purecd_stepsizes} simplifies into 
$n \tau_i \sigma \|M_{1, i}\|^2 = \gamma < 1$, $\forall i \in \{1, 2\}$.

We looked for a Lyapunov function involving, for $i \in \{1, 2\}$, the vectors $x_\star^i$, $y^*$, $M_{1, i} x_\star^i$, $
M_{1, i}^\top y_*$, $x_0^i$, $y_0$, $M_{1, i} x_0^i$, $M_{1, i}^\top y_0$, 
$(h^*)'(\bar y_1) \in \partial h^*(\bar y_1)$, $M_{1, i}^\top \bar y_1$, $g_i'(\bar x_1^i) \in \partial g_i(\bar x_1^i)$ and $M_{1,i} \bar x_1^i$ as well as the nonnegative quantities defined by $D_P(x_0) = \sum_i g_i(x_0^i) - g_i(x_*^i) + y_\star^\top M (x_0 - x_\star)$,
$D_D(y_0) = h(\bar y_1) - h(y_\star) - (\bar y_1 - y_\star)^\top M x_\star$, $\mathbb E[\|y_1 - y_\star\|^2]$ and $\mathbb E[\|x_1^i - x_\star^i\|^2]$. Moreover, we consider the expected difference between successive iterates as the measure of convergence $R(z_0, z_\star) = \mathbb E[ \sum_{i=1}^n \frac 1{\tau_i} \|x_0^i - x_1^i\|^2 + \frac{1}{\sigma} \|y_0 - y_1\|^2]$.

After compilation with PEPit and DSP-CVXPY, we obtain two semi-definite programs, each with 7 SDP constraints involving 1409 variables, 63 affine inequality constraints and more that 1000 equality constraints. These are then sent to the SCS solver~\cite{odonoghue:21} for numerical resolution.

For $\gamma =1.1$, which is beyond what is proved in \cite{alacaoglu2020random}, we get 
\[
\min_{\{V \in \mathcal Q\}} \max_{\{(f, g, L) \in \mathcal F, z \in \mathcal Z, z^* \text{ saddle point}\}} \mathbb E[V(\text{PURE-CD}(f,g,L, z))] + R(z) - V(z) \approx 0.5 > 0
\]
This shows that there does not exist a Lyapunov function for this algorithm which is a combination of quadratic terms in the vectors listed above and conic combination of the nonnegative quantities listed above. To get a more precise answer, we should either find a counter-example where the algorithm diverges or a Lyapunov function that will necessarily be more general.

For $\gamma = 0.9$, we get
\[
\min_{\{V \in \mathcal Q\}} \max_{\{(f, g, L) \in \mathcal F, z \in \mathcal Z, z^* \text{ saddle point}\}} \mathbb E[V(\text{PURE-CD}(f,g,L, z))] + R(z) - V(z) \approx 0
\]
which numerically confirms the convergence of PURE-CD.

In theory, since this is what is proved in \cite{alacaoglu2020random}, we could also obtain a more general Lyapunov function that does not require that $z_\star$ should be a saddle point. However, except for the case $n=1$, the solver struggles to solve the problem and terminates with an inaccurate solution.

\section{Conclusion}

In this paper, we showed how to combine several recent works on
performance estimation problems and numerical solvers in order
to discover automatically Lyapunov functions for optimization algorithms.
We then showed how this technique can be used to make discoveries 
on the properties of complex optimization algorithms. Firstly, we studied the linear convergence of primal-dual hybrid gradient under quadratic error bound of the smooth gap with a larger step size range than previous works. Secondly, we could conjecture and prove a slightly improved 
convergence rate result for randomized coordinate descent. Thirdly, we 
experienced the limits of the approach when studying a randomized primal-dual coordinate update method. 

This work paves the way to the use of performance estimation problems 
to investigate complex questions. Note however, that in order to really get results, we need semi-definite programming solvers able to accurately solve large and ill-conditioned problems. In future works, we plan to use the automatic Lyapunov function method to study algorithms for nonconvex-nonconcave saddle point problems.

\section*{Acknowledgment}

This work was supported by the Agence National de la Recherche grant ANR-20-CE40-0027, Optimal Primal-Dual Algorithms (APDO).

\bibliographystyle{alpha}
\bibliography{literature.bib}
	
\end{document}